\documentclass{amsart}
\usepackage{amsfonts}
\usepackage{amsmath,amssymb}	
\usepackage{amsthm}
\usepackage{amscd}
\usepackage{graphics}
\usepackage{graphicx}
{
\theoremstyle{definition}
\newtheorem{Def}{Definition}

}
\newtheorem{Cor}{Corollary}
\newtheorem{Prop}{Proposition}
\newtheorem{Thm}{Theorem}

\begin{document}
\title[Simple fold maps and manifolds bounded by the source manifolds]{Simple fold maps and compact manifolds bounded by their source manifolds}
\author{Naoki Kitazawa}
\address{Department of Mathematics, Tokyo Institute of Technology, 
2-12-1 Ookayama, Meguro-ku, Tokyo 152-8551, JAPAN, Tel +81-(0)3-5734-2205, 
Fax +81-(0)3-5734-2738,
}
\email{kitazawa.n.aa@m.titech.ac.jp}
\maketitle
\begin{abstract}
Fold maps are higher dimensional versions of Morse functions, which play important roles in the studies of smooth manifolds, and such
 general maps also have been fundamental tools in the studies of
 smooth manifolds by using generic maps. In this paper, we study {\it simple} fold maps, which are fold maps such that any connected component of the inverse image
 of each singular value includes at most
 one singular point. More precisely, we consider
 simple fold maps having simple structures locally or globally and show that the source manifolds are bounded
 by (PL) manifolds obtained by considering the structures of maps under appropriate conditions. Such studies are regarded
 as extensions of results obtained by Saeki, Suzuoka etc. by 2005, which state that closed manifolds admitting simple fold maps and more
 generally stable maps into manifolds of lower dimensions without boundaries inverse images of whose regular
 values are always disjoint unions of spheres are bounded by compact manifolds obtained by observing the given maps.

\end{abstract}

\section{Introduction, terminologies and notation}
\label{sec:1}
As a higher dimensional versions of Morse functions, {\it fold maps} have been fundamental tools in the studies of smooth manifolds by using generic maps. 
 A {\it fold map} is defined as a smooth map
 such that each singular point is of the form
$$(x_1, \cdots, x_m) \mapsto (x_1,\cdots,x_{n-1},\sum_{k=n}^{m-i}{x_k}^2-\sum_{k=m-i+1}^{m}{x_k}^2)$$ for two positive integers $m \geq n$ and an integer $0 \leq i \leq m-n+1$. A Morse function
 is naturally regarded as a fold map ($n=1$).
  For a fold map from a closed smooth manifold of dimension $m$ into a smooth manifold of dimension $n$
 (without boundary), the following two hold ($m \geq n \geq 1$).
\begin{enumerate}
\item The {\it singular set}, defined as the set of all the singular points, is a closed smooth submanifold of dimension $n-1$ of the source manifold. 
\item The restriction map to the singular set is a smooth immersion of codimension $1$.
\end{enumerate}

Studies of such maps were started by Whitney (\cite{whitney}) and Thom (\cite{thom}) in the 1950s.
We also note that if the restriction map to the singular set is an immersion with normal
 crossings, it is {\it stable} (stable maps are important in the theory of global singularity; see \cite{golubitskyguillemin} for example). \\ 

\ \ \ Since around the 1990s, fold maps with additional conditions have been
 actively studied. For example, in \cite{burletderham}, \cite{furuyaporto}, \cite{saeki2}, \cite{saekisakuma} and \cite{sakuma}, {\it special
 generic} maps, which are defined as smooth maps such that singular points are always of the form
$$(x_1, \cdots, x_m) \mapsto (x_1,\cdots,x_{n-1},\sum_{k=n}^{m}{x_k}^2)$$ for two positive integers $m \geq n$, were studied. {\it Simple} fold maps
 are defined as fold maps such that fibers of singular values do not have any connected component with more
 than one singular points (see cite \cite{saeki} and \cite{sakuma2}) and special generic maps are simple. In
 \cite{kobayashisaeki}, Kobayashi and Saeki
 investigated topology of {\it stable} maps including fold maps which are stable from closed manifolds of dimensions larger than $2$ into the plane. 
maps including fold maps which are stable from closed manifolds of dimensions larger than $2$ into the plane. 

\ \ \ Later, in \cite{saekisuzuoka}, Saeki and Suzuoka
 found good properties of manifolds admitting stable maps such that the inverse images of regular values are disjoint unions of spheres. The main
 theorem \cite[Theorem 4.1]{saekisuzuoka} states that a closed smooth manifold of
 dimension $4$ admitting such a stable map into a surface without boundary bounds a nice compact smooth manifold of dimension $5$ we can construct
 by observing the inverse images of the maps.
 As \cite[Lemma 1]{kitazawa}, the theorem is generalized as a proposition for a simple fold map such that the inverse images of regular values
 are disjoint unions of spheres from a closed smooth manifold of
 dimension $m$ into a smooth manifold of
 dimension $n$ without boundary (see also \cite[Proposition 3.12]{saeki} and its proof). In this paper, we consider extensions of these works. We consider
 simple fold maps such that the inverse images of regular values are more general and show that under appropriate differential
 topological conditions, the source manifolds are bounded by compact PL or smooth manifolds. 

\ \ \ This paper is organized as follows.

\ \ \ Section \ref{sec:2} is for preliminaries. We review the {\it Reeb space} of a smooth
 map, which is the space consisting of all the connected components of
 all the fibers of the map. We review fundamental facts on simple fold maps in Proposition \ref{prop:1}.

\ \ \ In section \ref{sec:3}, as main theorems, under some appropriate situations, we prove facts similar to ones in \cite{saekisuzuoka} introduced before for closed smooth manifolds admitting simple fold maps such that the inverse images of regular values
 are not always disjoint unions of spheres by analogy of the proofs of the original statements with extra
 new technique on algebraic and differential topology. More precisely, we
 first decompose
 the Reeb space $W_f$ of the given simple fold map $f:M \rightarrow N$ into
some pieces by using its simplicial structure, construct manifolds piece by piece and then we glue them together
 to obtain the desired manifold $W$ such that $\partial W=M$. As a result, we also obtain other objects such as a polyhedron simple homotopy equivalent to $W$, which is
 not obtained in the proofs of known results.

\ \ \ In this paper, smooth manifolds and smooth maps between them are of class $C^{\infty}$ unless otherwise stated. Throughout this paper, we assume that $M$ is a closed smooth manifold of dimension $m$ ($m \geq 1$)
 and that $N$ is a smooth manifold of dimension $n$ without boundary ($m \geq n \geq 1$). We set $f$ as a smooth map from $M$ into $N$ and
 denote the {\it singular set} of the map, defined as the set
 of all the singular points of the map, by $S(f)$. \\
 
\ \ \ We also note about homeomorphisms. In this paper, we often consider ${\rm PL}$ homeomorphisms between two
 polyhedra including ones between two ${\rm PL}$ manifolds. In this paper, for two polyhedra $X_1$ and $X_2$, a
 homeomorphism $\phi:X_1 \rightarrow X_2$ which is isotopic (in the topology category) to a homeomorphism from $X_1$ onto $X_2$ which gives an
 isomorphism between underlying simplicial complexes is said to be a {\it ${\rm PL}$ homeomorphism}. Note that any
 diffeomorphism between two diffeomorphic manifolds is
 regarded as a ${\rm PL}$ homeomorphism, where we consider the canonical ${\rm PL}$ structures.\\

\section{Preliminaries}
\label{sec:2}

\subsection{Reeb spaces}

We introduce the {\it Reeb space} of a continuous map.

\begin{Def}
\label{def:3}
 Let $X$, $Y$ be topological spaces. For $p_1, p_2 \in X$ and for a map $c:X \rightarrow Y$, 
 we define as $p_1 {\sim}_c p_2$ if and only if $p_1$ and $p_2$ are in
 the same connected component of $c^{-1}(p)$ for some $p \in Y$. ${\sim}_{c}$ is an equivalence relation.

\ \ \ We set the quotient space $W_c:=X/{\sim}_c$. we call $W_c$ the {\it Reeb space} of $c$.
\end{Def}

 We denote the induced quotient map from $X$ into $W_c$ by $q_c$. We define $\bar{c}:W_c \rightarrow Y$
 so that $c=\bar{c} \circ q_c$. 
For example, for a
 Morse function, the Reeb space
 is a graph and for a simple fold map, the Reeb space is homeomorphic to a polyhedron which is not so complex (see Proposition \ref{prop:1} later). For a
 special generic map,
 the Reeb space is homeomorphic to a smooth manifold (see section 2 of \cite{saeki2}).

\ \ \ Here, we introduce terms on spheres, fiber bundles and polyhedra which are important in this paper.

\ \ \ In this paper, an {\it almost-sphere} means a smooth homotopy sphere given by glueing same dimensional two standard discs together by a diffeomorphism between the boundaries. We note that the underlying ${\rm PL}$ manifold of an almost-sphere is a standard sphere.

\ \ \ We often use terminologies on (fiber) bundles in this paper (see also \cite{steenrod} for example). For a topological space $X$, an {\it $X$-bundle} is a bundle whose fiber is $X$.  A bundle
 whose structure group is $G$
 is a {\it trivia bundle} if it is equivalent to the product bundle as a bundle whose structure group is $G$. In this paper, a {\it PL {\rm (}smooth{\rm )} bundle} means a bundle
 whose fiber is a polyhedron (resp. smooth manifold) and whose structure group is a group consisting of some PL homeomorphisms (resp. diffeomorphisms) of the fiber. A {\it linear} bundle is
 a bundle whose fiber is a standard sphere or a standard disc and whose structure group consist of linear transformations on the fiber.

\ \ \ In this paper, we also use terminologies on polyhedra. In this paper, for a polyhedron of dimension $k \geq 1$, 
 a {\it branched} point means a point such that every open
 neighborhood of the point is not homeomorphic to any open set of ${\mathbb{R}}^k$ or ${{\mathbb{R}}^k}_{+}:=\{(x_1,\cdots,x_k) \in {\mathbb{R}}^k \mid x_k \geq 0\}$. If a polyhedron $X$ of dimension
 $k$ does not have branched points, then it is a manifold with triangulation and we can define the {\it interior} ${\rm Int} X$ and the {\it boundary} $\partial X$.

\ \ \ The following proposition is well-known and we omit the proof. 

\begin{Prop}
\label{prop:1}
Let $f:M \rightarrow N$ be a special generic map, a simple fold map, or a stable fold map. Then, $W_f$ is regarded as a polyhedron and the following statements hold. 
\begin{enumerate}
\item
\label{prop:1.1}
 $W_f-q_f(S(f))$ is uniquely regarded as a smooth manifold of dimension $n$ such that $q_f:M-S(f) \rightarrow W_f-q_f(S(f))$ is a smooth submersion.
 For any compact subset $P$ of any connected component of $W_f-q_f(S(f))$, is regarded as the total space of a bundle given by $q_f {\mid}_{{q_f}^{-1}(P)}:{q_f}^{-1}(P) \rightarrow P$.
    
\item
\label{prop:1.2}
 Suppose that $f$ is simple. Then, for any connected component $C$ of $S(f)$, there exists a small
 regular neighborhood $N(q_f(C))$ of $q_f(C)$ in $W_f$ such that $N(q_f(C))$ is regarded as the total space of
 a PL bundle over $C$. 

 Furthermore, ${q_f}^{-1}(N(q_f(C)))$ is regarded as the total space
 of a PL bundle given by the composition of the map ${q_f} {\mid}_{{q_f}^{-1}(N(q_f(C)))}:{q_f}^{-1}(N(q_f(C))) \rightarrow N(q_f(C))$ and the projection of
 the bundle $N(q_f(C))$ over $C$.

\item 
\label{prop:1.3}

If $C$ is a connected component consisting of singular points of index $0$, then the bundle $N(q_f(C))$ before is
 a trivial $[0,1]$-bundle over $q_f(C)$ and $q_f(C)$ is the image of a section
 of the bundle over $N(q_f(C))$ corresponding to the point $0 \in [0,1]$. Furthermore, the bundle ${q_f}^{-1}(N(q_f(C)))$ is a
 linear $D^{m-n+1}$-bundle over $q_f(C)$. 

\item
\label{prop:1.4}

If in the situation of Proposition \ref{prop:1} (\ref{prop:1.2}), $C$ is a connected component of the singular set consisting of singular points of index larger than $0$, then the bundle $N(q_f(C))$ before is
 a PL $[-1,1]$-bundle over $q_f(C)$ and $q_f(C)$ is the image of a section
 of the bundle over $N(q_f(C))$ corresponding to $0 \in [-1,1]$. This bundle is trivial if and only if $N(q_f(C))-q_f(C)$ is not connected.

\item
\label{prop:1.5}

Ifin the situation of Proposition \ref{prop:1} (\ref{prop:1.2}), $C$ is a connected component of the singular set such that $q_f(C)$ consists of branched points, then the bundle $N(q_f(C))$ before is
 a PL bundle over $q_f(C)$ whose fiber is $K:=\{re^{i \theta} \mid 0 \leq r \leq 1, \theta=0, \frac{2}{3}\pi, \frac{4}{3}\pi \}$ and whose structure group
 consists of only the identity map and the conjugation $z \mapsto \bar{z}$. Furthermore, $q_f(C)$ is the image of a section
 of the bundle over $N(q_f(C))$ corresponding to $0 \in K$. 
 
\end{enumerate}  
\end{Prop}

\section{Simple fold maps and compact manifolds bounded by the source manifolds of the maps}
\label{sec:3}

As main works, in this section, we consider simple fold maps satisfying appropriate conditions on regular fibers and extra algebraic and differential topological
 conditions and construct manifolds bounded by their source manifolds. We review results on simple fold maps and stable maps whose regular fibers are disjoint
 unions of spheres and then, we prove Theorems \ref{thm:1}-\ref{thm:3}.

\subsection{Simple fold maps whose regular fibers are disjoint unions of spheres}
\label{subsec:3.1}

We review \cite[Theorem 4.1]{saekisuzuoka}.

\begin{Prop}[\cite{saekisuzuoka}]
\label{prop:2}
 Let $f:M \rightarrow N$ be a simple fold map from a closed manifold of dimension $4$ into a manifold of dimension $2$ without boundary. For
 each regular value $p$, let $f^{-1}(p)$ be a disjoint union of finite copies of $S^2$.

\ \ \ Then, there exist a compact smooth manifold $W$ of dimension $5$ such that $\partial W = M$ and
a continuous map $r: W \rightarrow W_f$ such that $r \mid_{\partial W}$ coincides with $q_f:M \rightarrow W_f$ and the following five hold.

\begin{enumerate}
\item For each $p \in W_f-q_f(S(f))$, $r^{-1}(p)$ is diffeomorphic to $D^3$.
\item $\bar{f} \circ r$ is a smooth submersion.
\item There exist a smooth triangulation of $W$ and a triangulation of $W_f$ such that $r$ is a simplicial map.
\item For each $p \in W_f$, $r^{-1}(p)$ collapses to a point and $r$ is a homotopy equivalence.
\item $W$ collapses to a subpolyhedron ${W_f}^{\prime}$ such that
 $r {\mid}_{{W_f}^{\prime}}:{W_f}^{\prime} \rightarrow W_f$ is a ${\rm PL}$ homeomorphism.
\end{enumerate}

If $M$ is orientable, then we can construct $W$ as an orientable manifold.
\end{Prop}

As a corollary to the proposition, we can prove the following.

\begin{Cor}[\cite{saekisuzuoka}]
\label{cor:1}
In the situation of Proposition \ref{prop:2}, let $M$ be connected and $i:M \rightarrow W$ be
 the natural inclusion. Then,
$${q_f}_{\ast}=r_{\ast} \circ i_{\ast}:{\pi}_k(M) \rightarrow {\pi}_k(W_f)$$
 gives an isomorphism for $k=0,1$.  
\end{Cor}

See also \cite[Lemma 1]{kitazawa2}, which is a similar statement for simple fold maps whose regular fibers are disjoint unions of spheres
 between manifolds of arbitary dimensions.

\subsection{More general simple fold maps and compact manifolds bounded by their source manifolds}
\label{subsec:3.2}

We study more general simple fold maps. Before precise studies, we define several technical conditions on the maps, which we often pose in the present paper, to obtain theorems.  

\begin{Def}
\label{def:4}
 Let $f:M \rightarrow N$ be a simple fold map sastisfying $m>n \geq 1$.
\begin{enumerate}
\item For each connected component $C$ of the set $q_f(q_f(S(f)))$, let $N(C)$ be a small regular neighborhood such
 that ${q_f}^{-1}(N(C))$ is regarded as the total space of a bundle over $C$. We
 call the bundle ${q_f}^{-1}(N(C))$ a {\it monodromy representation} on $C$ and if we can take this as a topologically trivial bundle, then $f$ is said to
 {\it have a topologically trivial monodromy} on $C$ and the bundle is said to be a {\it topologically trivial
 monodromy bundle} on $C$. If $f$ has a topologically trivial monodromy on $C$ for each component $C$, then $f$
 is said to {\it have topologically trivial monodromies on the singular part}. 

We can replace "topologically" by "PL"
 and "smooth" and define similar terminologies.
\item Let $N(q_f(S(f)))$ be a small regular neighborhood of the set $q_f(S(f))$. For any connected component $R$ of
 the set $W_f-{\rm Int} N(q_f(S(f)))$, ${q_f}^{-1}(R)$ is regarded
 as the total space of a bundle over $R$. We
 call the bundle ${q_f}^{-1}(R)$ a {\it monodromy representation} on $R$ and if we
 can choose the bundle ${q_f}^{-1}(R)$ as a topologically trivial bundle, then $f$ is said to
 {\it have a topologically trivial monodromy} on $R$ and the bundle is said to be a {\it topologically trivial monodromy bundle} on $R$. If $f$ has a
 topologically trivial monodromy on $R$ for each component $R$, then $f$
 is said to {\it have topologically trivial monodromies on the regular part}.

We can replace "topologically" by "PL"
 and "smooth" and define similar terminologies.
\end{enumerate}
\end{Def}

\begin{Thm}
\label{thm:1}
Let $M$ be a closed and connected manifold of dimension $m$, $f:M \rightarrow {\mathbb{R}}^n$ be a simple fold
 map and $m>n \geq 1$ hold. Suppose that $f$ has topologically trivial monodromies on the singular part and the regular part and that the following conditions hold.

\begin{enumerate}
\item For each connected component $C$ of the set $q_f(S(f))$, we can take a topologocally trivial monodromy bundle on $C$ as a bundle whose fiber is a compact topological manifold ${F^{\prime}}_C$ with non-empty boundary.

\item Let $N(C)$ be a small regular neighborhood of each connected component $C$ as in Definition \ref{def:4} and let $N(q_f(S(f)))$ be the disjoint
 union of $N(C)$ for every $C$. For any connected component $R$ of
 the set $W_f-{\rm Int} N(q_f(S(f)))$, we can take a topologically trivial monodromy bundle on $R$ as a bundle whose fiber is a topological manifold $F_R$ bounding a compact topological manifold $E_R$.

\item For each connected component $C$ of the set $q_f(S(f))$, let us denote the family of all the connected components of
 the set $W_f-{\rm Int} N(q_f(S(f)))$ intersecting $N(C)$ by $\{R_{C_{\lambda}}\}$. Then the boundary $\partial {F^{\prime}}_C$ of the
 manifold ${F^{\prime}}_C$ before
 is represented as the disjoint union of closed topological manifolds
 $\{F_{C_{\lambda}}\}$ {\rm (}$\lambda \in \Lambda${\rm )}, where $F_{C_{\lambda}}$ is homeomorphic
 to $F_{R_{C_{\lambda}}}$. Furthermore, the subbundle corresponding to the boundary $\partial {F^{\prime}}_C \subset F^{\prime}$ of a topologically
 trivial monodromy bundle on $C$ is trivial.  

\end{enumerate}

 Then, we can construct a topologically trivial bundle $W_R$ over
 $R$ whose fiber is a compact topological manifold $E_R$ satisfying so that the topologically trivial monodromy bundle over
 $R$ before is a subbundle of this bundle correspondig to the boundary. Moreover as a result, we can construct a
 topologically trivial bundle over $C$ whose fiber is a closed topological manifold $F_C$ obtained by attaching
 the manifolds $\{F_{C_{\lambda}}\}$ {\rm (}$\lambda \in \Lambda${\rm )} on the boundary of ${F^{\prime}}_C$. 
 
  Furthermore, assume also
 that $F_C$ bounds a compact topological manifold $E_C$. Then, there exist a compact topological manifold $W$ of dimension $m+1$ bounded by $M$.

\ \ \ Especially, let similar conditions hold even if we replace "topologically" by "PL". Then, we can construct $W$ as
 a PL manifold and we can obtain a polyhedron $V$ whose dimension is smaller than $W$ and continuous maps $r:W \rightarrow V$
 and $s:V \rightarrow W_f$ and triangulations of $W$, $V$ and $W_f$ so that the following conditions hold.
 
\begin{enumerate}
\item $r$ and $s$ are simplicial.
\item If $p$ is in a connected component in $W_f-{\rm Int} N(q_f(S(f)))$, then $s^{-1}(p)$ is a subpolyhedron of $V$, $(s \circ r)^{-1}(p)$ collapses
 to $s^{-1}(p)$ and the dimension of $s^{-1}(p)$ is smaller than that of $(s \circ r)^{-1}(p)$.
\item For each connected component $C$ of $N(q_f(S(f)))$, there exists a PL subbudle $V_C$ of the trivial PL $E_C$-bundle over $C$ which is trivial and whose fiber
 is a subpolyhedron of $E_C$, simple homotopy equivalent to $E_C$ and of dimension smaller than $E_C$. 
\item $W$ collapses to a subpolyhedron ${V}^{\prime}$ such that $r {\mid}_{{V}^{\prime}}:{V}^{\prime} \rightarrow V$ is a PL homeomorphism.
\end{enumerate}

In addition, let similar conditions hold even if we replace "PL" by "smooth". Then, we can take the manifold $W$ as a
 smooth manifold and we can obtain similar polyhedron, continuous maps and smooth triangulations.

\end{Thm}

\begin{proof}

We can easily obtain a topological manifold $W$ by gluing the following ($m+1$)-dimensional manifolds.

\begin{itemize}
\item The total space $W_R$ of the topologically trivial $E_R$-bundle over $R$ explained in the assumption.
\item The total space $W_C$ of a topologically trivial $E_C$-bundle over $C$ such that the trivial $F_C$-bundle over $C$ explained in
 the assumption is a subbundle of this bundle satisfying $\partial E_C=F_C$.
\end{itemize}

We can obtain $W$ as a PL (smooth) manifold in the PL (resp. smooth) case.

We show additional statements in the PL case and by using a similar method we
 can show them in the smooth case.   

For any connected component $R$ of $W_f-{\rm Int} N(q_f(S(f)))$, by considering a handle decomposition of $E_R$ in the
 PL category, we obtain a polyhedron which $E_R$ collapses to and whose dimension
 is smaller than that of $E_R$ and thus, we obtain a
 trivial PL subbundle $V_R$ of the trivial PL bundle $W_R$ over
 $R$ whose fiber is a PL manifold $E_R$. $V_R$ is a polyhedron whose dimension is smaller than that of $W_R$ and we obtain
 a polyhedron ${{V}^{\prime}}_R$, continuous maps $r_R:W_R \rightarrow V_R$
 and $s_R:V_R \rightarrow R$ and triangulations of $W_R$, $V_R$, ${{V}^{\prime}}_R$ and $R$ so that they satisfy desired conditions.

For any connected component $C$ of the set $q_f(S(f))$, by considering a handle decomposition
 of $E_C$ in the
 PL category, we obtain a polyhedron which $E_C$ collapses to and whose dimension
 is smaller than that of $E_C$ and thus, we obtain a
 trivial PL subbundle $V_C$ of the trivial PL bundle $W_C$ over
 $R$ whose fiber is a PL manifold $E_C$. We can take $V_C$ so that we can
 attach $V_C \bigcap \partial W_C$ with the disjoint union of $V_{R_{C_{\lambda}}} \bigcap \partial W_{R_{C_{\lambda}}}$ for every $C_{\lambda}$ compatiably together. 

 This completes the proof.

\end{proof}

We have the following corollary similar to Corollary \ref{cor:1}, which we prove later with Corollary \ref{cor:3}.

\begin{Cor}
\label{cor:2}
In the situation of Theorem \ref{thm:1}, let $i:M \rightarrow W$ be the natural inclusion. Then, the
 homomorphism $r_{\ast} \circ i_{\ast}:{\pi}_j(M) \rightarrow {\pi}_j(V)$ is an isomorphism for $0 \leq j \leq m-\dim{V}-1$.
\end{Cor}

As specific cases, we have the following two statements.

\begin{Thm}
\label{thm:2}
Let $m,n \in \mathbb{N}$ and let $m>n \geq 1$ hold. Let $M$ be a closed and connected manifold of dimension $m$ and $f:M \rightarrow {\mathbb{R}}^n$ be a simple fold
 map. Suppose that $f$ has topologically trivial monodromies on the singular part.
\begin{enumerate}
\item
\label{thm:2.1}
 Suppose that $f$ has topologically trivial monodromies on the regular part. We also assume that for each closed manifold of dimension $m-n$ or $m-n+1$, there
 exists a compact topological manifold bounded by it. Then, we can construct a compact topological manifold $W$ of dimension $m+1$ bounded by $M$ in
 Theorem \ref{thm:1}. In the PL case, we can construct $W$ as
 a PL manifold and we can obtain a polyhedron $V$ whose dimension is smaller than $W$ and continuous maps $r:W \rightarrow V$
 and $s:V \rightarrow W_f$ and triangulations of $W$, $V$ and $W_f$ so that the four conditions listed in Theorem \ref{thm:1} hold. In the smooth case, in
 addition, we can construct $W$ as a smooth manifold. 
\item
\label{thm:2.2}
 Suppose that the fiber of the $F_C$-bundle over $C$ in Theorem \ref{thm:1} is a sphere. Then, we can construct a compact topological manifold $W$ of dimension $m+1$ bounded by $M$ in Theorem \ref{thm:1}. In the PL
 and smooth cases, we obtain objects similar to those of Theorem \ref{thm:1} and (\ref{thm:2.1}) of this theorem. 
\end{enumerate}
\end{Thm}

\begin{proof}
The first statement follows immediately from Theorem \ref{thm:1}. We can prove the second statement by noticing that a sphere bundle is naturally a
 subbundle of a disc bundle whose dimension is larger than that of the sphere bundle by one.    
\end{proof}


 We introduce and prove Theorem \ref{thm:3}, which is regarded as an extension of a
 specific case of Theorem \ref{thm:2} (\ref{thm:2.2}). We define the {\it mapping cylinder} of a continuous map $c:X \rightarrow Y$ as
 the quotient space of $(X \times [0,1]) \sqcup Y$ obtained by identifying $(x,1)$ and $c(x)$ ($x \in X$).       

\begin{Thm}
\label{thm:3}
Let $M$ be a closed manifold of dimension $m$, $N$ be a manifold of dimension $n$ without
 boundary and $f:M \rightarrow N$ be a simple fold map. Let $n \geq 1$ and $m-n \geq 2$. 
\begin{enumerate}
\item For any point $p \in W_f-q_f(S(f))$, ${q_f}^{-1}(p)$ is an almost-sphere of dimension $m-n$ or PL homeomorphic to the product of two standard spheres
 whose dimensions do not coincide. If a connected component $R$ of $W_f-q_f(S(f))$ contains a point $p$ such that ${q_f}^{-1}(p)$ is an almost-sphere, then $R$ is said
 to be an {\rm AS-region}. 

Furthermore, if
 ${q_f}^{-1}(p)$ is not an almost-sphere, then each connected component of the boundary of the closure of the connected
 component of $W_f-q_f(S(f))$ containing $p$ contains no branched point and it is also the boundary of the closure of an AS-region. 

\item Let $C$ be a connected component of the set $S(f)$ such that $q_f(C)$ consists of branched points, then every connected
 component of $W_f-q_f(S(f))$ such that the boundary of its closure contains $C$ as a connected component is an AS-region. 
\item Let $C$ be a connected component of the set $S(f)$ and let $N(C)$ be a small tubular neighborhood in Proposition \ref{prop:1}(\ref{prop:1.3}). Let $0<k<m-n$ be an integer and let the fiber of the bundle ${q_f}^{-1}(N(C))$ be
 PL homeomorphic to an ($m-n+1$)-dimensional manifold
 simple homotopy equivalent to $S^{m-n-k}$ with the interior of a smoothly embedded ($m-n+1$)-dimensional stadard closed disc removed. Then, for
 any point $p$ in the connected
 component $R_C$ of $W_f-q_f(S(f))$ such that the boundary of its closure contains $C$ as a connected component, ${q_f}^{-1}(p)$ is PL
 homeomorphic to $S^k \times S^{m-n-k}$ and we can take the monodromy representation on $R_C$ as a PL $S^k \times S^{m-n-k}$-bundle over $R_C$ whose
 structure group consists of PL homeomorphisms regarded as bundle isomorphisms on this trivial $S^{m-n-k}$-bundle $S^k \times S^{m-n-k}$ over $S^{m-n-k}$. In this
 situation,
 $R_C$ is said to be a {\rm $k$ S-region}. 

 Furthermore, each connected component of the boundary of the closure of a $k$ S-region satisfies the condition before which $C$ satisfies. Such
 connected components of the boundary of the closure of a $k$ S-region are said to be {\rm $k$ S-loci}.
\end{enumerate}
\ \ \ Then, there exist a compact PL manifold $W$ of dimension $m+1$ such that $\partial W = M$, a polyhedron $V$
 and continuous maps $r: W \rightarrow V$ and $s:V \rightarrow W_f$ and the following two hold.
\begin{enumerate}
\item There exist a triangulation of $W$, a triangulation of $V$ and a triangulation of $W_f$ such that $r$ is a simplicial map and that $s$
 is a simplicial map and the following three hold.
\begin{enumerate}
\item For each $p \in V$, $r^{-1}(p)$ collapses to a point and $r$ is a homotopy equivalence.
\item If $p$ is in the closure of an AS-region in $W_f$, then $s^{-1}(p)$ is a point. If $p$ is in an AS-region in $W_f$ and $q \in s^{-1}(p)$, then
 $r^{-1}(q)$ is ${\rm PL}$ homeomorphic to $D^{m-n+1}$. 
\item If $p$ is in a $k$ S-region in $W_f$ whose closure is bounded by a disjoint union of $k$ S-loci, then $s^{-1}(p)$ is PL homeomorphic to $S^k$. If
 $p$ is in a $k$ S-region in $W_f$ whose closure is bounded by a disjoint union of $k$ S-loci and $q \in s^{-1}(p)$, then $r^{-1}(q)$ is PL homeomorphic
 to $D^{m-n-k+1}$.
\end{enumerate}
\item $W$ collapses to a subpolyhedron ${V}^{\prime}$ such that $r {\mid}_{{V}^{\prime}}:{V}^{\prime} \rightarrow V$ is a PL homeomorphism.
\end{enumerate}

If $M$ is orientable, then we can construct $W$ as an orientable manifold.

\end{Thm}

\begin{proof}
We construct a compact PL manifold of dimension $m+1$ bounded by $M$. We denote the set of all the singular points of index $i$ by $F_i(f)$. \\
 \\
Step 1 \ \ \ Around a regular neighborhood of $q_f(F_0(f))$.

 $q_f(F_0(f))$ is the image of the set $F_0(f)$ of all the definite fold points of $f$. Let $N(q_f(F_0(f)))$ be a small
 regular neighborhood of $q_f(F_0(f))$ as in Proposition \ref{prop:1} (\ref{prop:1.3}). $N(q_f(F_0(f)))$ is regarded as the total space of a trivial PL
 bundle over $q_f(F_0(f))$ with the fiber $[0,1]$. We may assume that $q_f(F_0(f))$ corresponds to the $0$-section ($0 \in [0,1]$). 

\ \ \ For each $p \in q_f(F_0(f))$, set $K_p:={q_f}^{-1}(\{p\} \times [0,1])$ for a fiber $\{p\} \times [0,1]$ of the bundle $N(q_f(F_0(f)))$ over $q_f(F_0(f))$. It
 is diffeomorphic to $D^{m-n+1}$. $q_f^{-1}(N(q_f(F_0(f))))$ is regarded as the total space of a linear bundle over $q_f(F_0(f))$ whose fiber is $D^{m-n+1}$ and $K_p$ is the fiber over $p \in q_f(F_0(f))$.
 

\ \ \ We obtain the following objects. See also the proof of Lemma 1 of \cite{kitazawa}.

\begin{enumerate}
\item A linear $D^{m-n+2}$-bundle $V_0$
 over $q_f(F_0(f))$ (we denote by $V_p$ the fiber over $p \in q_f(F_0(f))$) such that $q_f^{-1}(N(q_f(F_0(f))))$ is regarded as the total
 space of a subbundle of the bundle $V_0$ with the fiber $D^{m-n+1}$ and that $q_f^{-1}(N(q_f(F_0(f))))$ is in the
 boundary of $V_0$.
\item A simplicial map $r_0:V_0 \rightarrow N(q_f(F_0(f)))$.
such that ${r_0}^{-1}(p,t)$ is ${\rm PL}$ homeomorphic to $D^{m-n+1}$ for $p \in q_f(F_0(f))$ and $t \in (0,1]$ and
 that ${r_0}^{-1}(p,t)$ is a point for $p \in q_f(F_0(f))$ and $t=0$. 
\item A subbundle $\widetilde{N(q_f(F_0(f)))} \subset V_0$ of dimension $n$ such that the followings hold.
\begin{enumerate}
\item $\widetilde{N(q_f(F_0(f)))}$ is of dimension $n$.
\item ${r_0} {\mid}_{\widetilde{N(q_f(F_0(f)))}}:\widetilde{N(q_f(F_0(f)))} \rightarrow N(q_f(F_0(f)))$ is
 a ${\rm PL}$ homeomorphism (a bundle isomorphism between the two PL bundles).
\item $\widetilde{N(q_f(F_0(f)))} \bigcap \partial V_p$ consists of two points $(p,0),(p,1) \in \{p\} \times [0,1]$. One
 of the two points is in $q_f^{-1}(N(q_f(F_0(f))))$ and the other point is not. 
\item $V_0$ collapses to $\widetilde{N(q_f(F_0(f)))}$.
\end{enumerate}
\end{enumerate}  

\ \ \ Let $B(f)$ be the set of all the branched points of $W_f$. By considering the attachments of handles, $B(f) \subset q_f(F_1(f))$ follows. \\
 \\
Step 2 \ \ \ A regular neighborhood of $B(f)$.

  Let $N(B(f))$ be a small
 regular neighborhood of $B(f)$. From Proposition \ref{prop:1} (\ref{prop:1.5}), $N(B(f))$ is regarded as the total space
 of a PL bundle over $B(f)$ whose fiber is $K:= \{r \exp (i \theta) \in \mathbb{C} \mid 0 \leq r \leq 1, \theta = 0, \frac{2}{3}\pi, \frac{4}{3}\pi \}$. We may
 assume that $B(f)$ corresponds to the $0$-section ($0 \in K$).

  For each $p \in B(f)$, set $K_p:={q_f}^{-1}(\{p\} \times K)$ for a fiber $\{p\} \times K$ of the bundle $N(B(f))$ over $B(f)$. It is PL homeomorphic to
 $S^{m-n+1}$ with the interior of a union of disjoint three {\rm (}$m-n+1${\rm )}-dimensional closed standard discs removed. We
 may assume that $q_f^{-1}(N(B(f)))$ is regarded as the total space of a smooth bundle over $Q(f)$ with a fiber PL homeomorphic to $S^{m-n+1}$ with the interior
 of a union of disjoint three standard closed {\rm (}$m-n+1${\rm )}-discs removed and $K_p$ is
 the fiber over $p \in B(f)$. \\
%
\ \ \ We obtain the following objects. See also the proof of Lemma 1 of \cite{kitazawa}.

\begin{enumerate}
\item A PL $D^{m-n+2}$-bundle $V_B$ over $B(f)$ (we
 denote by $V_p$ the fiber over $p \in B(f)$) such
 that the bundle $q_f^{-1}(N(B(f)))$ is a subbundle of the bundle $V_B$ and that ${q_f}^{-1}(N(B(f)))$ is in the boundary of $V_B$.
\item A simplicial map
 $r_B:V_B \rightarrow N(B(f))$
 such that ${r_B}^{-1}(p,t)$ is PL homeomorphic to $D^{m-n+1}$ for $p \in B(f)$ and $t \in K-\{0\}$ and
 that ${r_B}^{-1}(p,t)$ collapses to a point for $p \in B(f)$ and $t=0$.
\item A subpolyhedron $\widetilde{N(B(f))} \subset V_B$ of dimension $n$ such that the following four hold.
\begin{enumerate}
\item $\widetilde{N(B(f))}$ is regarded as the total space of a subbundle of the bundle $V_1$.
\item ${r_B} {\mid}_{\widetilde{N(B(f))}}:\widetilde{N(B(f))} \rightarrow N(B(f))$ is
 a PL homeomorphism (a bundle isomorphism between the two PL bundles).
\item $\widetilde{N(B(f))} \bigcap \partial V_p$ consists of three points $(p,\exp(\frac{2}{3} k \pi i)) \in \{p\} \times K$ ($k=0,1,2$) and is not in ${q_f}^{-1}(N(B(f)))$. Furthermore, each connected
 component of $\partial V_p-q_f^{-1}(N(B(f)))$ includes one of these points and $V_B$ collapses to $\widetilde{N(B(f))}$. 
\item $V_B$ collapses to $\widetilde{N(B(f))}$.
\end{enumerate}
\end{enumerate}

For $2 \leq k \leq m-n$, let $G_k(f) \subset q_f(S(f))$ be the disjoint union of all the $k-1$ S-loci. By the assumption, we can
 take a regular neighborhood $N(G_k(f))$ so that ${q_f}^{-1}(N(G_k(f)))$ is regarded as the total space of a smooth bundle over $G_k(f)$ whose fiber
 is PL homeomorphic to an ($m-n+1$)-dimensional manifold simple homotopy equivalent to $S^{m-n-k+1}$ with the interior of a smoothly embedded
 ($m-n+1$)-dimensional standard closed disc removed.
 As mentioned in Proposition \ref{prop:1} (\ref{prop:1.4}), $N(G_k(f))$
 is regarded as the total space of a trivial PL $[-1,1]$-bundle over $G_k(f)$ and $G_k(f)$ is the image of the section corresponding to the point $0 \in [-1,1]$. \\
 \\ 
Step 3 \ \ \ Around the set $W_f-{\rm Int} ((N(q_f(F_0(f))) \sqcup N(B(f))) \sqcup {\sqcup}_{k} N(G_k(f)))$.

\ \ \ This set is a compact manifold of dimension $n$ and let $\{R_{\lambda}\}_{\lambda \in \Lambda}$ be
 the family of all the connected components of the set. $R_{\lambda}$ is in an AS-region or an S-region.

\ \ \ Let $R_{\lambda}$ be in an AS-region. Then, ${q_f}^{-1}(R_{\lambda})$ is regarded as the total space of a smooth bundle
 over $R_{\lambda}$ whose fiber is ${\rm PL}$ homeomorphic to $S^{m-n}$. We can define a PL $D^{m-n+1}$-bundle $V_{R_{\lambda}}$ over $R_{\lambda}$ which is an associated bundle of
 the bundle ${q_f}^{-1}(R_{\lambda})$ over $R_{\lambda}$ (we regard $S^{m-n}=\partial D^{m-n+1}$) by a
 PL map $r_{R_{\lambda}}:V_{R_{\lambda}} \rightarrow {R_{\lambda}}$. In addition, we take the associated bundle so that the structure group
 is a group consisting of
 PL homeomorphisms $r$ on $D^{m-n+1}$ such that $r(0)=0$ and that for a PL homeomorphism ${r}^{\prime}$ on $S^{m-n}$, $\frac{r(x)}{|x|}={r}^{\prime}(\frac{x}{|x|})$ ($x \neq 0$). Let $\widetilde{P(R_{\lambda})} \subset V_{R_{\lambda}}$ be the
 section of the associated bundle corresponding to the point $0 \in D^{m-n+1}$. Finally we set $P(R_{\lambda}):=R_{\lambda}$ and
 let $s_{R_{\lambda}}:P(R_{\lambda}) \rightarrow R_{\lambda}$ be the identity map.

\ \ \ Let $R_{\lambda}$ be in a $k$ S-region whose closure is bounded
 by a disjoint union of $k$ S-loci. Then, ${q_f}^{-1}(R_{\lambda})$ is regarded as the total space of a smooth bundle
 over $R_{\lambda}$ whose fiber is ${\rm PL}$ homeomorphic to $S^k \times S^{m-n-k}$. We can take a PL ($S^k \times D^{m-n-k+1}$)-bundle $V_{R_{\lambda}}$ over ${R_{\lambda}}$ which is
 an associated bundle of the bundle ${q_f}^{-1}(R_{\lambda})$ over $R_{\lambda}$ (we regard $S^k \times S^{m-n-k}=S^k \times \partial D^{m-n-k+1}$). In addition ,we take the associated bundle so that the structure group
 is a group consisting of
 some PL homeomorphisms $r$ on $S^k \times D^{m-n-k+1}$ such that $r$ is a bundle isomorphism on
 a PL bundle $S^k \times D^{m-n-k+1}$ over $S^k$ whose structure group consists of PL homeomorphisms ${r_1}$ on $D^{m-n-k+1}$, where $r_1(0)=0$ and for a
 PL homeomorphism ${r_2}$ on $S^{m-n-k}$, $\frac{r_1(x)}{|x|}={r_2}(\frac{x}{|x|})$ ($x \neq 0$) and that $r$ induces a PL homeomorphism
 on the base space $S^k$. Let $\widetilde{P(R_{\lambda})} \subset V_{R_{\lambda}}$ be the
 subbundle whose fiber is $S^k \times \{0\} \subset S^k \times D^{m-n-k+1}$. The bundle $V_{R_{\lambda}}$ over ${R_{\lambda}}$ is also given by the
 composition of two PL maps $r_{R_{\lambda}}:V_{R_{\lambda}} \rightarrow P(R_{\lambda})$ and $s_{R_{\lambda}}:P(R_{\lambda}) \rightarrow R_{\lambda}$, where
 $P(R_{\lambda})$ is a PL manifold and by $s_{R_{\lambda}}$ makes $P(R_{\lambda})$ a bundle over $R_{\lambda}$ equivalent to
 the bundle $\widetilde{P(R_{\lambda})}$ over $R_{\lambda}$.

\ \ \ Now we set disjoint unions $V_R:={\sqcup}_{\lambda \in \Lambda} V_{R_{\lambda}}$, $P_R:={\sqcup}_{\lambda \in \Lambda} P(R_{\lambda})$, $r_R:={\sqcup}_{\lambda \in \Lambda} r_{R_{\lambda}}$, $s_R:={\sqcup}_{\lambda \in \Lambda} s_{R_{\lambda}}$ and $\widetilde{P_R}:={\sqcup}_{\lambda \in \Lambda} \widetilde{P(R_{\lambda})}$. \\
 \\
Step 4 \ \ \ Around a regular neighborhood of $G_k(f)$ ($2 \leq k \leq m-n$)

\ \ \ After Steps 1, 2 and 3, for any connected component $C_k$ of $G_k(f)$ and a small regular neighborhood $N(C_k)$ of $C_k$ as in Proposition \ref{prop:1} (\ref{prop:1.4}), we obtain a PL bundle over $C_k$ whose fiber
 is an ($m-n+1$)-dimensional PL manifold $B_{C_k}$ simple homotopy equivalent to $S^{m-n-k+1}$ whose
 boundary $\partial B_{C_k}$ is $S^{k-1} \times S^{m-n-k+1}$ with the product manifold $D^{m-n-k+2} \times S^{k-1}$ attached by the product of a pair
 of PL homeomorphisms
 $({\phi}_{{C_1}_k}:\partial D^{m-n-k+2} \rightarrow S^{m-n-k+1}, {\phi}_{{C_2}_k}:S^{k-1} \rightarrow S^{k-1})$ and that the
 bundle ${q_f}^{-1}(N(C_k))$ is a subbundle of the
 $B_{C_k}$-bundle.

\ \ \ We obtain the following objects.

\begin{enumerate}
\item A PL $D^{m-n+2}$-bundle $V_{C_k}$ over $C_K$ (let $V_p$ denote the
 fiber over $p$ of the base space) such that the subbundle $\partial V_{C_k}$ corresponding to
 $\partial D^{m-n+2}=B_{C_k} {\bigcup}_{{\phi}_{{C_1}_k} \times {\phi}_{{C_2}_k}} (D^{m-n-k+2} \times S^{k-1})$ is the PL $S^{m-n+1}$-bundle
 in the previous paragraph. The bundle $q_f^{-1}(N(C_k))$
 over $C_k$ is
 a subbundle of the bundle $V_{C_k}$ and is in the boundary of $V_{C_k}$ (we regard it is also a subbundle of the $B_{C_k}$-bundle).
\item A PL bundle $P(C_k)$ whose fiber is the mapping cylinder of the constant
 map $c_k:S^{k-1} \rightarrow [0,1]$ satisfying $c_k(S^{k-1})=\{0\}$ and whose base space
 is $C_k$.
\item Simplicial maps $r_{C_k}:V_{C_k} \rightarrow P(C_k)$ and $s_{C_k}:P(C_k) \rightarrow N(C_k)$
 such that ${s_{C_k}}^{-1}(p)$ is a point for $p$ in the closure of an AS-region, that ${s_{C_k}}^{-1}(p)$ is PL homeomorphic to $S^{k-1}$
 for $p$ in a $k-1$ S-region, that ${r_{C_k}}^{-1}(q)$ is PL homeomorphic to $D^{m-n+1}$ for $q \in {s_{C_k}}^{-1}(p)$ for $p$ in an AS-region,
 that ${r_{C_k}}^{-1}(q)$ is PL homeomorphic to $D^{m-n-k+2}$ for $q \in {s_{C_k}}^{-1}(p)$ for $p$ in a $k-1$ S-region and that ${(s_{C_k} \circ {r_{C_k}})}^{-1}(\{p\} \times [-1,1])=V_p$ for all $p \in C_k$ ($\{p\} \times [-1,1]$ is a
 fiber of the bundle $N(C_k)$ over $C_k$; $N(C_k)$
 is regarded as the total space of a trivial PL $[-1,1]$-bundle over $C_k$ and $C_k$ is the image of the
 section corresponding to the point $0 \in [-1,1]$. Furthermore, ${r_{C_k}}^{-1}(q)$
 collapses to a point for $q \in P(C_k)$. 
\item A subpolyhedron $\widetilde{P(C_k)} \subset V_{C_k}$ of dimension $n+k-1$ such
 that the following four hold.
\begin{enumerate} 
\item $\widetilde{P(C_k)}$ is regarded as the total space of a subbundle of the bundle $V_{C_k}$.
\item ${r_{C_k}} {\mid}_{\widetilde{P(C_k)}}:\widetilde{P(C_k)} \rightarrow P(C_k)$ is
 a PL homeomorphism and a bundle isomorphism between the two PL bundles.
\item $\widetilde{P(C_k)} \bigcap \partial V_p$ does not contain any points in
 ${q_f}^{-1}(N(C_k))$ and is the disjoint union of a point in $B_{C_k}$ and
 a ($k-1$)-dimensional sphere $\{0\} \times S^{k-1} \subset D^{m-n-k+2} \times S^{k-1}$, where $B_{C_k}$
 and $D^{m-n-k+2} \times S^{k-1}$ are the naturally embedded submanifolds of the fiber $\partial V_p$ over $p$ of the bundle $\partial V_{C_k} \subset V_{C_k}$, which is PL homeomorphic
 to $B_{C_k} {\bigcup}_{{\phi}_{{C_1}_k} \times {\phi}_{{C_2}_k}} (D^{m-n-k+2} \times S^{k-1})$.
\item $V_{C_k}$ collapses to $\widetilde{P(C_k)}$.
\end{enumerate} 
\end{enumerate}

$V_G:={\sqcup}_{C_k} V_{C_k}$, $P_G:={\sqcup}_{C_k} P(C_k)$, $r_G:={\sqcup}_{C_k} r_{C_k}$, $s_G:={\sqcup}_{C_k} s_{C_k}$ and $\widetilde{P_G}:={\sqcup}_{C_k} \widetilde{P(C_k)}$.    

Thus, by gluing the manifolds $V_0$, $V_B$, $V_R$ and $V_G$ together, we obtain the desired manifold $V$. We can obtain other desired objects. This completes the proof.

\end{proof}

We have the following corollary.

\begin{Cor}
\label{cor:3}
In the situation of Theorem \ref{thm:3}, let $M$ be connected and $i:M \rightarrow W$ be
 the natural inclusion. Then \\
$$r_{\ast} \circ i_{\ast}:{\pi}_k(M) \rightarrow {\pi}_k(V)$$
 gives an isomorphism for $0 \leq k \leq m-\dim{V}-1$.
\end{Cor}

\begin{proof}[The proof of Corollaries \ref{cor:2} and \ref{cor:3}]
 Since $r$ is a homotopy equivalence, we have only
 to show that $i_{\ast}:{\pi}_{k}(M) \rightarrow {\pi}_{k}(W)$ ($0 \leq k \leq m-\dim{V}-1$)
is an isomorphism. Since $W$ collapses to a polyhedron of dimension $\dim{V}$, it is regarded as a polyhedron consisting
 of handles whose indices are not larger than $\dim{V}$. By dualizing the handles, $V$ is
 regarded as a PL manifold obtained by attaching handles whose indices are not smaller than $m-\dim{V}+1$ along
the component $M \times \{1\}$ of $M \times [0,1]$. Hence, the homomorphism $i_{\ast}:{\pi}_{k}(M) \rightarrow {\pi}_{k}(W)$ ($0 \leq k \leq m-\dim{V}-1$) is an isomorphism and this completes the proof.
\end{proof}

\end{document}